\documentclass[12pt]{amsart} \usepackage{amssymb}
\usepackage{graphicx}

\theoremstyle{plain}
\newtheorem{theorem}{Theorem}
\newtheorem{lemma}[theorem]{Lemma}

\theoremstyle{definition}
\newtheorem{definition}[theorem]{Definition}

\newcommand{\und}{\underline}

\begin{document}
\title{Constraints, MMSNP and expander relational structures}

\author{G\'abor Kun}

\email{kungabor@cs.elte.hu}

\address{Department of Algebra and Number Theory, E\" otv\" os University,
Budapest, P\' azm\' any s\' et\' any 1C, H-1117, Hungary} 

\begin{abstract}
We give a poly-time construction for a combinatorial classic 
known as Sparse Incomparability Lemma, studied by Erd\H os, Lov\'asz, 
Ne\v{s}et\v{r}il, R\"odl and others: We show that every Constraint
Satisfaction Problem is poly-time equivalent to its restriction to 
structures with large girth.
This implies that the complexity classes CSP and Monotone Monadic Strict 
NP introduced by Feder and Vardi are computationally equivalent. The 
technical novelty of the paper is a concept of expander relations and 
a new type of product for relational structures: a generalization of 
the zig-zag product, the twisted product.
\end{abstract}

\thanks{This research was supported by OTKA Grants no. T043671 and NK 67867,
Subhash Khot's NSF Waterman Award CCF-1061938 and the MTA R\'enyi "Lend\"ulet" 
Groups and Graphs Research Group.}

\maketitle

\section{Introduction}

The construction of graphs with large girth and chromatic number is a classic 
in probabilistic combinatorics. Many great mathematicians have contributed to 
this: Erd\H os \cite{Erdos} gave a probabilistic construction for graphs. 
Lov\'asz \cite{Lo} had a deterministic, but huge construction for hypergraphs. 
Ne\v{s}et\v{r}il and R\"odl gave a short probabilistic construction \cite{NR} 
for hypergraphs, see also Duffus, R\"odl, Sands and Sauer\cite{DRSS}. 
Feder and Vardi showed \cite{FV} a more 
general statement known as Sparse Incomparability Lemma: They proved that for 
every CSP problem there is a randomized, poly-time algorithm that transforms 
every input structure of the CSP to an equivalent one of large girth. 

Ne\v{s}et\v{r}il and Matou\v{s}ek gave a deterministic poly-time algorithm 
\cite{MN} to this in case of graphs. This was simplified in the recent work of 
Ne\v{s}et\v{r}il and Siggers \cite{NS}. The main result of this paper is a 
deterministic, poly-time algorithm for the Sparse Incomparability Lemma for
general CSP's. \footnote{We will follow the terminology of relational 
structures, but our theorems hold in the special case of hypergraphs, too.}

\begin{theorem}~\label{reduction}(Algorithm)
Let $t,k$ be positive integers and $\tau$ a finite relational type.
For every structure ${\bf S}$ of type $\tau$ there exists a polynomial 
time constructible structure ${\bf S'}$ of type $\tau$ with girth $>k$ 
such that for every structure ${\bf T}$ of size $<t$ the equivalence 
${\bf S} \in CSP({\bf T}) \iff {\bf S'} \in CSP({\bf T})$ holds.
Moreover, ${\bf S'} \in CSP({\bf S})$.
\end{theorem}

This theorem also answers a problem posed by Ne\v{s}et\v{r}il, Kostochka 
and Smol\'{\i}kova \cite{KNS}. However, our work was primarily motivated 
by the paper of Feder and Vardi \cite{FV} on the dichotomy conjecture: They 
analyzed the complexity classes CSP and Monotone Monadic Strict NP (MMSNP). 
The class MMSNP contains the class CSP, and it has much bigger
expressive power. Feder and Vardi proved that
these classes are equivalent in a random sense. (For more on these classes 
and a simple proof see \cite{KNEJC}.) The only random part 
in their algorithm comes from their probabilistic proof of the Sparse 
Incomparability Lemma, so we can derandomize their result using 
Theorem~\ref{reduction}.

\begin{theorem}~\label{CSP=MMSNP}
Let $\tau$ be a finite relational type, $L \subseteq Rel(\tau)$ an
$MMSNP$ language. Then there is a finite set of relational
structures $\mathcal{T} \subset Rel(\tau)$ such that

\begin{enumerate}

\item $L$ has a polynomial time reduction to $CSP(\mathcal{T})$.

\item $CSP(\mathcal{T})$ has a polynomial time reduction to $L$.

\end{enumerate}

\end{theorem}


Note that the equivalence of the complexity classes CSP and MMSNP
does not only mean that both of these classes contain an NP-complete
problem. In particular, Theorem~\ref{CSP=MMSNP} shows that if
dichotomy holds for CSP then it also holds for MMSNP. 


Ne\v{s}et\v{r}il and Matou\v{s}ek used expander graphs to give a poly-time 
algorithm for the Sparse Incomparability Lemma in the case of graphs.
Expander graphs are sparse but highly connected graphs. These 
play an important role in number theory, group theory and graph theory.  
Ajtai, Koml\'os and Szemer\'edi used expanders in their paper on parallel 
sorting \cite{AKSz}: this was the first time when expanders were used in
computer science. "Optimal expander graphs", Ramanujan graphs were 
constructed by Margulis \cite{M} and independently by Lubotzky, 
Phillips and Sarnak \cite{LPS}. Simpler and simpler constructions were 
found in the last decade \cite{ASS, RVW}. Recently Lubotzky, Samuels and 
Vishne introduced a concept of Ramanujan complexes \cite{LSV,LSV2}.

On the other hand, for relational structures (hypergraphs) no similar
construction or even definition is known. We introduce a concept of
{\it expander relations}. We say that the $r$-ary relation $R$ on
$S$ is an $\varepsilon$-expander relation if for every $S_1, \dots ,S_r
\subseteq S$ the number of relational tuples with the $i$th
coordinate in $S_i$ differs by less than $\varepsilon |R|$ from the
expected value. We construct $\varepsilon$-expander relational structures 
with large girth and bounded degree in poly-time.

\begin{theorem}\label{alg}(Algorithm)
Let $\tau$ be a finite relational type, $k$ a positive integer and
$\varepsilon>0$. Then for every $n>n_{\tau,\varepsilon,k}$ there 
exists a polynomial time constructible $\varepsilon$-expander 
${\bf S}$ of size $n$, type $\tau$, maximal degree at most  
$M=M_{\tau,\varepsilon}$ and girth at least $k$.
\end{theorem}
 
In order to give this construction we define the {\it twisted product} 
of relational structures, a generalization of the so-called zig-zag 
product used by Reingold, Vadhan and Widgerson \cite{RVW}. Alon, Schwartz 
and Shapira used a similar product called replacement product in their 
expander construction \cite{ASS}.

In Section 2 we give the basic definitions. Section 3 contains the
novelties of this paper: the definition and properties of
expander relational structures and the twisted product. In Section 4 we  
construct expander relational structures with large girth and
bounded degree. In Section 5 we prove Theorem~\ref{reduction} and 
Theorem~\ref{CSP=MMSNP}.

\section{Definitions, notations}

We will work with finite relational structures throughout this paper: 
we denote these by boldface letters ${\bf A,B,C}, \dots$ and their base 
set by $A,B,C, \dots$, respectively. For an $r$-ary relational symbol $R$
and relational structure ${\bf A}$ with base set $A$ let $R=R({\bf A})$       
denote the set of tuples of ${\bf A}$ which are
in relation $R$. Recall, that a homomorphism is a mapping which
preserves all relations. Just to be explicit, for relational
structures ${\bf A,\bf B}$ of the same type $\tau$ a mapping $f: A
\longrightarrow B$ is a {\it homomorphism} ${\bf A} \longrightarrow
\bf B$ if for every $r$-ary relational symbol $R \in \tau$ and
$(x_1, \ldots, x_r) \in A^r$ the implication $(x_1, 
\ldots, x_r) \in R(\bf A) \longrightarrow 
(f(x_1), \dots, f(x_r)) \in R(\bf B)$ holds. 
A {\it cycle} in a relational structure 
${\bf A}$ is either a minimal sequence of distinct points and distinct
tuples $x_0, r_1, x_1, \ldots, r_t, x_t = x_0$ where $t \geq 2$, each 
tuple $r_i$ belongs to one of the relations $R(\bf A)$ and each $x_i$ 
is a coordinate of $r_i$ and $r_{i+1}$, or, in the degenerate case, a 
relational tuple with at least one multiple coordinate. The
{\it length} of the cycle is the integer $t$ in the first case and
$1$ in the second case.
The {\it girth} of a structure $\bf A$ is the length of the shortest
cycle in $\bf A$ (if it exists; otherwise it is a forest and we 
define the girth to be infinity). The
degree of an element $x$ of ${\bf S}$ is the number of relational
tuples containing $x$ (with multiplicity). Denote the maximal degree
in $\bf S$ by $\Delta(\bf S)$. Given a function $f:S \rightarrow
\mathbb{R}$ let $|f|=\sum_{x \in S} |f(x)|$ denote its first norm
and $max (f)$ its maximum, respectively.

For the relational structure ${\bf A}$ set $CSP({\bf A}) = \{ {\bf
B}: {\bf B} \text{ is homomorphic}$ \\ to ${\bf A} \}$. 
For a finite set $\mathcal{T}$ of relational structures of the same type 
define $CSP(\mathcal{T})= \cup_{{\bf A} \in \mathcal{T}} CSP({\bf A})$.
The class CSP consists of languages of the form $CSP(\mathcal{T})$. 


\section{Expander relations and the twisted product}

\begin{definition}
Given a finite relational structure ${\bf A}$, a relation $R
\subseteq A^r$ and functions $f_1, \dots ,f_r: A
\rightarrow \mathbb{R}$ let us denote the sum

$\begin{displaystyle}
\sum_{(x_1, \dots ,x_r) \in R({\bf A})} \prod_{i=1}^r f_i(x_i)
\end{displaystyle}$ 

\noindent
by $R(f_1, \dots ,f_r)$.
For the subsets $S_1, \dots ,S_r \subseteq A$ set
$R(S_1, \dots ,S_r) = R(\chi_{S_1}, \dots ,\chi_{S_r})$. 
This equals the number of $r$-tuples $(x_1, \dots ,x_r) \in R$
such that $x_1 \in S_1, \dots ,x_r \in S_r$.
\end{definition}


\begin{definition}
A nonempty $r$-ary relation $R \subseteq S^r$ is called
an $\varepsilon$-expander relation if for every 
$S_1, \dots ,S_r \subseteq S$ the inequality

\noindent $\Big| R(S_1, \dots ,S_r) - |R| \frac{\prod_{i=1}^r
|S_i|}{|S|^r} \Big| \leq \varepsilon |R|$ holds.

\noindent A relational structure ${\bf S}$ is a $(\Delta, \varepsilon)$
-expander relational structure if every at least binary
relation of ${\bf S}$ is an $\varepsilon$-expander relation and
$\Delta({\bf S}) \leq \Delta$.
\end{definition}

An expander graph is an expander relational structure: this is a 
trivial consequence of the Expander Mixing Lemma \cite{NatiAvi}.
We postpone the study of such expanders until Section 4. Now we 
give several equivalent definitions for expander relations.

\begin{lemma}~\label{eqdef}
For a finite $r$-ary relation $R \subseteq S^r$ the followings are
equivalent.
\begin{enumerate}

\item For every $f_1, \dots ,f_r: S \rightarrow [0;\infty)$,

\noindent $\Big| R(f_1, \dots ,f_r) - |R| \frac{\prod_{i=1}^k
|f_i|}{|S|^r} \Big| \leq \varepsilon |R| \prod_{i=1}^r max (f_i)$
holds.

\item For every $f_1, \dots ,f_r: S \rightarrow [0;1]$,

\noindent $\Big| R(f_1, \dots ,f_r) - |R| \frac{\prod_{i=1}^r
|f_i|}{|S|^r} \Big| \leq \varepsilon |R|$ holds.

\item $R$ is an $\varepsilon$-expander relation.
\end{enumerate}
\end{lemma}

\begin{proof}
The implication $(1) \rightarrow (2)$ is trivial. $(3)$ is the
special case of $(2)$ when all the functions $f_i$ are
characteristic functions. We have to prove $(3) \rightarrow
(1)$: 

\noindent
$\Big| R(f_1, \dots ,f_r) - \prod_{i=1}^r |f_i| \frac{|R|}{|S|^r} \Big|= \\
\Big| \int_{y_1=0}^{max (f_1)} \dots \int_{y_r=0}^{max (f_r)} \sum_{
(e_1,\dots ,e_r) \in R} \prod_{i=1}^r \chi_{ \{y_i<f_i(e_i) \} }
d y_1 \dots d y_r - \\
\int_{y_1=0}^{max (f_1)} \dots \int_{y_r=0}^{max (f_r)}
\frac{|R|}{|S|^r} \prod_{i=1}^r \Big( \sum_{s \in S} \chi_{ \{
y_i<f_i(s) \} } \Big)
d y_1 \dots d y_r \Big| \leq \\
\int_{y_1=0}^{max (f_1)} \dots \int_{y_r=0}^{max (f_r)}
\Big| R(\{s: y_1<f_1(s) \}, \dots ,\{ s: y_r<f_r(s) \} )- \\
\frac{|R|}{|S|^r} \prod_{i=1}^r | \{ s: y_i<f_i(s) \} | \Big| d y_1 
\dots d y_r \leq \\
\int_{y_1=0}^{max (f_1)} \dots \int_{y_r=0}^{max (f_r)} \varepsilon
|R| d y_1 \dots d y_r = \varepsilon |R| \prod_{i=1}^r max (f_i)$.
\end{proof}

\begin{definition}
Let ${\bf A}$ and ${\bf B}$ be relational structures of type $\tau$.
We say that ${\bf C}$ is a twisted product of ${\bf A}$ and ${\bf
B}$ if the followings hold.

\begin{enumerate}

\item
The base set of ${\bf C}$ is the product set: $C=A \times B$.

\item
The projection $\pi_B: A \times B \rightarrow B$ is a homomorphism
$\bf C \rightarrow B$.

\item
For every $r$-ary relational symbol $R$ of type $\tau$, $1 \leq i
\leq r$ and relational tuple $t=(t_1, \dots ,t_r) \in R({\bf B})$
there exists a bijection $\alpha_{t,i}: A \rightarrow C$ such that
$\pi_B \circ \alpha_{t,i} =t_i$ and $(x_1, \dots ,x_r) \in R({\bf
A}) \iff \big( \alpha_{t,1}(x_1), \dots ,\alpha_{t,r}(x_r) \big) \in
R({\bf C})$.

\end{enumerate}
\end{definition}

If all the bijections in the definition are identical we get the 
direct product ${\bf A} \times {\bf B}$. 
In the case of simple, undirected graphs the last condition means
that the preimage of every edge in ${\bf B}$ is isomorphic to the
direct product of ${\bf A}$ and an edge. The celebrated zig-zag
product \cite{RVW} is a very special case (e.g. ${\bf A}$ is a
complete graph with loops). Two structures may have many different
twisted products: we can choose many bijections freely.
However, every twisted product of two expanders is an expander.

\begin{lemma}~\label{expszor}
Consider an $\varepsilon_A$-expander ${\bf A}$ and an
$\varepsilon_B$-expander ${\bf B}$ of type $\tau$. If ${\bf C}$ is
the twisted product of $A$ and $B$ then $C$ is an $(\varepsilon_A +
\varepsilon_B)$-expander. And $\Delta({\bf A}) \Delta({\bf B}) \geq
\Delta({\bf C})$ holds for the maximal degrees.
\end{lemma}

\begin{proof}
Let $R$ be an at least binary relation of type $\tau$. We will
prove that $(2)$ of Lemma~\ref{eqdef} holds for $R({\bf C})$.
Consider the functions $f_1, \dots ,f_r: C \rightarrow [0;1]$. Let
$g_i: B \rightarrow \mathbb{R}$ denote the function $g_i(b) = \sum_{x
\in \pi_B^{-1}(b)} f_i(x)$. Now $|g_i|=|f_i|$, and for every $b \in
B$ the inequality $0 \leq g_i(b) \leq |A|$ holds. So the expander
property of $R({\bf B})$ implies that

\noindent $\Big| R({\bf B})(g_1, \dots ,g_r)- |R({\bf B})
|\frac{|g_1| \dots |g_r|}{|B|^r} \Big| \leq \varepsilon_B |R({\bf
B})| |A|^r$.

\noindent Given an $r$-tuple $\und{b} = (b_1, \dots ,b_r) \in R({\bf
B})$ consider the bijections $\alpha_{\und{b},1}, \dots
,$\\ $\alpha_{\und{b},r}$ determining the twisted product. Clearly
$g_i(b_i)= \big| f_i|_{\pi_B^{-1}(b_i)} \big|= \big|f_i \circ
\alpha_{\und{b},i} \big|$. We sum up all the error terms using
$|A||B|=|C|, |R({\bf A})||R({\bf B})|=|R({\bf C})|$ and the triangle
inequality.

\medskip

\noindent $\begin{displaystyle} \Big| R({\bf C})(f_1, \dots ,f_r) -
|R({\bf C})| \frac{\prod_{i=1}^r
|f_i|}{|C|^r} \Big|= \\
\Big| R({\bf C})(f_1, \dots ,f_r) - |R({\bf C})|
\frac{\prod_{i=1}^r |g_i|}{|C|^r} \Big| \leq \\
\Big| R({\bf C})(f_1, \dots ,f_r)-\frac{|R({\bf A})|}{|A|^r}
R({\bf B})(g_1, \dots ,g_r) \Big| + \\
\Big| \frac{|R({\bf A})}{|A|^r} R({\bf B})(g_1, \dots ,g_r)-
|R({\bf C}) |\frac{ \prod_{i=1}^r |g_i|}{|C|^r} \Big| = \\
\Big| \sum_{\und{b} \in R({\bf B})} \Big( R({\bf A})(f_1 \circ
\alpha_{\und{b},1}, \dots ,f_r \circ \alpha_{\und{b},r}) -
\frac{|R({\bf A})|}{|A|^r} \prod_{i=1}^r g_i(b_i) \Big) \Big| + \\
\frac{|R({\bf A})|}{|A|^r} \Big| R({\bf B})(g_1, \dots ,g_r)-
|R({\bf B})|\frac{ \prod_{i=1}^r |g_i|}{|B|^r} \Big| \leq \\
\sum_{\und{b} \in R({\bf B})} \Big( \Big| R({\bf A})(f_1 \circ
\alpha_{\und{b},1}, \dots ,f_r \circ \alpha_{\und{b},r})-
\frac{|R({\bf A})|}{|A|^r}\prod_{i=1}^r |f_i \circ
\alpha_{\und{b},i}|
\Big| \Big)+ \\
\frac{|R({\bf A})|}{|A|^r} \varepsilon_B |R({\bf B})| |A|^r \leq
\sum_{\und{b} \in R({\bf B})} \varepsilon_A |R({\bf A})|
+\varepsilon_B |R({\bf C})| =(\varepsilon_A+\varepsilon_B) |R({\bf
C})|
\end{displaystyle}.$

The statement about the maximal degrees follows immediately from the
definition.
\end{proof}

Now we have arrived at the crucial technical theorem of the paper: 
Two structures with small maximal degree have a twisted product 
with large girth if the first structure has large girth.

\begin{theorem}(Algorithm)~\label{karcsu}
Consider the finite relational structures $\bf A$ and $\bf B$ of type 
$\tau$. Suppose that the girth of $\bf A$ is $\geq k$ and $|A|^{1/k} >
\Delta({\bf A}) \Delta({\bf B})$. Then there exists a twisted
product ${\bf C}$ of ${\bf A}$ and ${\bf B}$ with girth $\geq k$.
The structure ${\bf C}$ can be constructed in polynomial time (in
$|{\bf A}|$ and $|{\bf B}|$).
\end{theorem}

\begin{proof}
We will define better and better twisted products of ${\bf A}$ and
${\bf B}$. The number of cycles of minimal length will
decrease in every step. We start with the direct product 
${\bf C}_0={\bf A \times B}$. Let ${\bf C}_i$ denote the twisted 
product after Step i. of the
algorithm. The bijections determining ${\bf C}_i$ are denoted by
$\alpha^i_{v,l}: A \rightarrow C$ (where $v$ is an $r$-ary
relational tuple of ${\bf B}$, $1 \leq l \leq r$ and $\pi_B \circ
\alpha_{v,l}=v_l$). 

Now we describe Step (i+1).
Assume that the girth of ${\bf C}_i$ is $j < k$. Let $t \in R({\bf
B})$ be an $r$-ary relational tuple, $1 \leq l \leq r$ and $c,c' \in
{\bf C}_i$ such that their distance is $\geq k$ and $\pi_B(c)=
\pi_B(c')=t_l$. We will specify other conditions on the choice of
$t,l,c$ and $c'$ later.

Now we will change the bijection $\alpha^i_{t,l}$ but none of the
other bijections defining ${\bf C}_i$. The relations of ${\bf
C}_{i+1}$ and ${\bf C}_i$ will agree but the $l^{th}$
coordinate of the tuples in $\pi_B^{-1}(t)$.

Set $\alpha^{i+1}_{t,l} = (a a') \circ \alpha^i_{t,l}$, where $(a
a')$ is the transposition of $A$ flipping $a$ and $a'$.

The figure illustrates this idea on the product of two
undirected paths. The number of cycles of length four decreases. 
(We neglect the fact that undirected graphs have many degenerate cycles 
of length two when considered as relational structures (digraphs). So 
we actually do not work with the shortest cycles in the figure.)

\begin{figure}[!h]
\begin{center}
\includegraphics[width=8.0cm]{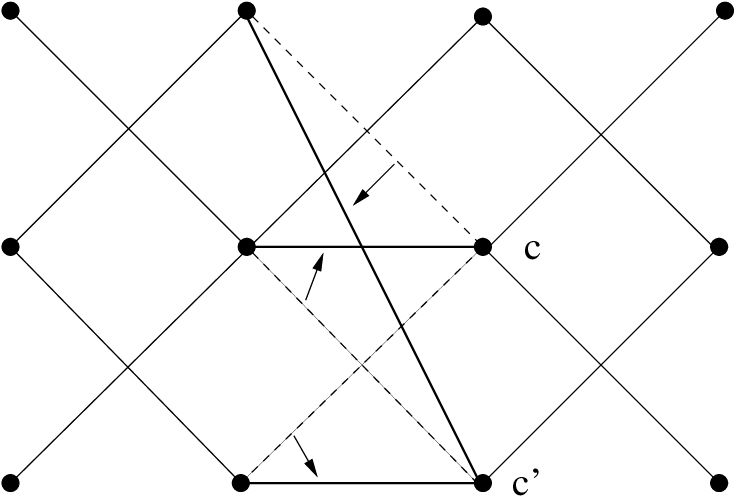}
\end{center}
\end{figure}

We call a cycle {\it short} if its length is $j$.
We will prove that the number of short cycles is strictly
less in ${\bf C}_{i+1}$ than in ${\bf C}_i$. We correspond to short
cycles in ${\bf C}_{i+1}$ short cycles in ${\bf C}_i$. Let $\xi$
denote the following bijection between the set of relational tuples
of ${\bf C}_{i+1}$ and ${\bf C}_i$. For a relational tuple $u$ of
${\bf C}_{i+1}$ set $\xi(u)=u$ if $\pi_B(u) \neq t$, else
$\xi(u)_h=u_h$ if $h \neq l$ and $\xi(u)_l= (a a') \circ u_l$.

We will show that the image of a short cycle under $\xi$ will be short.
Call the tuple $u$ {\it critical} if $\xi(u) \neq u$. This means that
$\pi_B(u) = t$ and the $l^{th}$ coordinate of $u$ is $c$ or
$c'$. The $l^{th}$ coordinate of a critical tuple is called {\it critical 
coordinate}.
We call a pair of tuples $(u_1, u_2)$ in ${\bf C}_{i+1}$ a {\it cutting pair}
if $\pi_B(u_1) \neq \pi_B(u_2)$ and the tuples intersect. Every
cycle with length $\leq k$ has a cutting pair: otherwise its image under 
$\pi_A$ would be a cycle, too.

\smallskip

{\bf Claim:} Let $t_1, \dots ,t_m$ be a cycle in ${\bf C}_{i+1}$, where 
$m < max \{ k, 2j \}$ and $(t_o,t_{o+1})$ is a cutting pair. Assume that
$c$ ($c'$) is a coordinate of both $t_o$ and $t_{o+1}$. Then $c$
($c'$) can not be a critical coordinate of $t_o$ or $t_{o+1}$. 

\smallskip

The Claim implies that the image of a short cycle under $\xi$ is a 
(short) cycle: If the image of two intersecting tuples under $\xi$
will not intersect then one should be a critical tuple and its 
critical coordinate should be in the intersection.

\smallskip

\begin{proof} (of the Claim)
We will prove by contradiction. We might suppose that $o=m$, say
$t_1$ is critical, $t_m$ is not and the critical coordinate $c$ 
is in their intersection. If there is no other adjacent 
critical-noncritical pair of tuples s.t. $c'$ is the critical coordinate
and it is in their intersection then the path $\xi(t_1), \dots ,\xi(t_m)$
connects $c$ and $c'$, hence $m \geq k$.

If there is such a pair, say $(t_h, t_{h+1})$ then we distinguish 
two cases:

If $t_{h+1}$ is critical then $\xi(t_1), \dots ,\xi(t_h)$ is a cycle, since 
$\xi(t_1)$ and $\xi(t_h)$ contains $c'$ and the other tuples remained 
adjacent under $\xi$. Similarly, $\xi(t_{h+1}), \dots ,\xi(t_m)$ is a 
cycle, hence $min \{h, m-h \} \geq j$, so $m \geq 2j$, a contradiction.
 
In the other case, when $t_h$ is critical the path $\xi(t_1), \dots ,\xi(t_h)$
connects $c'$ and $c$, hence $k \leq h < m$, a contradiction.

\end{proof}

The main step of the algorithm does not increase the girth. 
The image of a short cycle under $\xi$ will not be a cycle 
if it has a cutting pair such that exactly one of the two 
tuples is critical and the critical coordinate is in the 
intersection. The image of the other short cycles is still 
a cycle, and the cutting pairs are the same.

Let us do the main step of the algorithm for all possible triple $c
\in {\bf C}, l, t$ (and arbitrary $c'$) such that $\pi_B(c)=t_l$.
This will hit every cycle of length $<max\{ k, 2j \}$, since the cutting
pairs of a cycle do not change. If we iterate this $(log_2(k)+1)$ times 
then we will get the required girth.

The number of such triples is $O(|{\bf B}| |A|)$. We need to find $c'$
in every step: this requires $O({\bf |A||B|})$ time using Breadth First 
Search. We can exchange $c$ and $c'$ in the appropriate tuples in the 
same time. Altogether, the running time of the algorithm is 
$O(|A|{\bf |A||B|^2}log (k) )$.
\end{proof}

\section{Construction of expanders with large girth}

We prove Theorem~\ref{alg} in this section.
First we give a probabilistic existential proof 
in the spirit of \cite{Erdos,FV}.

\begin{lemma}~\label{random}
Let $\tau$ be a finite relational type, $k$ a positive integer and
$\varepsilon>0$. Then there is a $\Delta > 0$ such that for every
$n$ large enough there exists a $(\Delta, \varepsilon)$-expander of
type $\tau$ on $n$ vertices with girth $\geq k$.
\end{lemma}

\begin{proof}
We consider a probability space on the set of relational structures
with base set $\{1, \dots ,2n \}$. For every $r$-ary relational
symbol $R \in \tau$ and $r$-tuple $u$ let $Pr(u \in R)
=\frac{D}{n^{r-1}}$ independently, where the constant $D$ will be
chosen later. The expected number of cycles with length $\leq k$ is
$O(c^kD^k)$, and the expected degree of a vertex is $O(cD)$, 
where $c$ is a constant depending only on $\tau$. 
Set $\Delta$ to be ten times the expected
value of the degree of a vertex.

The Markov inequality implies that the number of elements covered by
the cycles with length $\leq k$ is at most $n/4$ with probability
$1-o_n(1)$, and the number of elements with degree at least $\Delta$ is
at most $n/4$ with probability $\geq \frac{3}{5}$. Remove every 
element with large degree or covered by a short cycle (to get a
structure on exactly $n$ elements we may remove more), and consider
the resulted structure ${\bf A}$ with base set $A$. With probability 
$\frac{3}{5}-o_n(1)$ the girth of ${\bf A}$ is $\geq k$ and the maximal 
degree of ${\bf A}$ is bounded. 

We have to prove
the expander property. Consider the $r$-ary relation $R$ and the
subsets $S_1, \dots S_r \subseteq \{1, \dots ,2n\}$. The probability
that $|R(S_1, \dots ,S_r) - Dn \frac{\prod_{i=1}^r |S_i|}{n^r}| <
\frac{\varepsilon}{3} Dn$ is at most $2e^{-\frac{1}{36}
\varepsilon^2 Dn}$ by the Chernoff bound. Since the number of the
possible choices is $2^{2nr}$ this will hold for a $D$ large enough
with probability $1-o_D(1)$ for every $r$, every $r$-ary relational 
symbol $R \in \tau$ and every $S_1, \dots ,S_r \subseteq A$. In particular, 
$\big| |R({\bf A})|-Dn \big| < \frac{\varepsilon}{3} Dn$. Hence ${\bf A}$ 
is an $\varepsilon$-expander. Altogether, with probability 
$\frac{3}{5}-o_n(1)-o_1(D)$ the structure ${\bf A}$ is a 
$(\Delta, \varepsilon)$-expander with girth $\geq k$.
And this probability is positive if $n$ and $D$ are large enough.
\end{proof}

\begin{lemma}~\label{graph2structure}
Consider the $d$-regular undirected graph $G=(V,E)$ with second largest
eigenvalue $\lambda$ and the integer $k \geq 2$. Let ${\bf S}$ be
the relational structure with base set $V$ and a single $k$-ary
relation $R_k$:

\noindent $R_k=\{ (a_1, \dots ,a_k) : \forall i \text{  }
(a_i,a_{i+1}) \in E \}$.

\noindent Then the relational structure ${\bf S}$ is a
$\Big( k d^{k-1}, (k-1)\frac{|\lambda|}{d} \Big)$-expander.
\end{lemma}
\begin{proof}
Note that $|R_i|= |V| d^{i-1}$ and the degree of every element is
$i d^{i-1}$. Set $\varepsilon=\frac{|\lambda|}{d}$. 
We prove by induction on $k$: First suppose that $k=2$. We will use
the expander mixing lemma \cite{NatiAvi}: for every $T,W \subseteq V$ the
inequality $|E(T,W)-d\frac{|T||W|}{|V|}| \leq \lambda \sqrt{|T||W|}$
holds. This implies $\Big| R_2(T,W)-\frac{|T||W|}{|V|^2} |R_2|  \Big|
\leq \varepsilon |R_2|$. Hence $R_2$ is an $\varepsilon$-expander
relation.

Assume that we have proved the lemma for $(k-1)$. Consider the
functions $x_1, \dots ,x_k: S \rightarrow [0;1]$. By
Lemma~\ref{eqdef} we need to show that $\Big| R_k(x_1, \dots ,x_k) -
|R_k| \frac{\prod_{i=1}^k |x_i|}{|V|^k} \Big| \leq \varepsilon (k-1)
|R_k|$.

For $i=1, \dots ,k$ define the sequence of functions $y_i: S
\rightarrow \mathbb{R}$ recursively. Let $y_0$ be the constant
$\frac{1}{d}$ function and $\begin{displaystyle} y_{i+1}(a) =
\sum_{(a,b) \in E} x_i(b) \end{displaystyle}$. Note that
$|y_{i+1}|=R_2(y_i,x_{i+1})$. Clearly $0 \leq max (y_i) \leq
d^{i-1}$ and $R_i(x_1, \dots ,x_i) = |y_i|$. Now we use the
inductional hypothesis:

\noindent
$\big| R_k(x_1, \dots ,x_k) -|R_k| \frac{\prod_{i=1}^k |x_i|}{|V|^k}
\big|=\Big| R_2(y_{k-1},x_k)- |R_k| 
\frac{\prod_{i=1}^k |x_i|}{|V|^k} \Big| \leq \\
\Big| R_2(y_{k-1},x_k) - |R_2| \frac{|y_{k-1}| |x_k|}{|V|^2} \Big|+
\Big| |R_2| \frac{|y_{k-1}| |x_k|}{|V|^2} - |R_k|
\frac{\prod_{i=1}^k |x_i|}{|V|^k} \Big| \leq  \\
\varepsilon |R_2| max(y_{k-1})+ \frac{d|x_k|}{|V|} \Big|
|y_{k-1}|-\frac{\prod_{i=1}^{k-1}|x_i|}{|V|^{k-1}}|R_{k-1}|\Big| \leq \\
\varepsilon |R_k|+\frac{d|x_k|}{|V|}(k-2)\varepsilon |R_{k-1}| \leq
(k-1)\varepsilon |R_k|$

\noindent The structure ${\bf S}$ is a
$\big( k d^{k-1}, (k-1)\varepsilon \big)$-expander.
\end{proof}


\begin{proof}(of Theorem~\ref{alg})
Assume that every relational symbol in $R$ 
is at most $r$-ary. We know that for some $d$ there exists a polynomial
time construction of $d$-regular expander graphs with eigenvalue gap
$|\frac{\lambda}{d}|<\frac{\varepsilon}{2r}$, see \cite{M, LPS}. 

On the other hand by Lemma~\ref{random} there exists an
$\frac{\varepsilon}{2}$-expander ${\bf A}$ with girth at least $k$
such that $|A|^{\frac{1}{k}}>r d^{r-1} \Delta({\bf A})$ holds. If $n$
is large enough then there exists such an ${\bf A}$ of size $log(n)$
by Theorem~\ref{random}, and so we can find it in polynomial time.

We construct an expander graph $G$ of size $\frac{n}{|A|}$ with the
above properties. Lemma~\ref{graph2structure} shows how to construct
an $\varepsilon \frac{r-1}{2r}$-expander ${\bf B}$ on the vertex set
of $G$ with maximal degree $rd^{r-1}$. The conditions of
Lemma~\ref{karcsu} hold for ${\bf A}$ and ${\bf B}$, hence there
exists a polynomial time constructible twisted product ${\bf C}$ of
${\bf A}$ and ${\bf B}$ with girth at least $k$. Now ${\bf C}$ is an
$\varepsilon$-expander by Lemma~\ref{expszor} with maximal degree at
most $M=rd^{r-1}\Delta(A)$ and girth at least $k$.
\end{proof}


\section{CSP vs MMSNP}

Now we prove Theorem~\ref{CSP=MMSNP} showing that CSP and MMSNP are
computationally equivalent. Feder and Vardi \cite{FV} proved the 
following (see \cite{KNEJC} for a simple proof). 

\begin{theorem}
Let $L$ be an $MMSNP$ language. Then there is a finite set of
relational structures $\mathcal{T}$ and a positive integer $k$ such
that

\begin{enumerate}

\item $L$ has a polynomial time reduction to $CSP(\mathcal{T})$.

\item $CSP(\mathcal{T})$ restricted to structures with girth at least $k$
has a polynomial time reduction to $L$.

\end{enumerate}
\end{theorem}

Theorem~\ref{reduction} and Theorem~\ref{FeV} would imply 
Theorem~\ref{CSP=MMSNP}. So we succeed to prove Theorem~\ref{reduction}.

\begin{lemma}~\label{exphely}
Consider the structures ${\bf A,B}$ and ${\bf T}$ of type $\tau$,
where ${\bf A}$ is an $\varepsilon$-expander. Suppose that every
relational symbol in $\tau$ is at most $r$-ary and
$\varepsilon|T|^r<1$. Let ${\bf C}$ be a twisted product of ${\bf
A}$ and ${\bf B}$. Then ${\bf B}$ is homomorphic to ${\bf T}$ iff
${\bf C}$ is homomorphic to ${\bf T}$.
\end{lemma}

\begin{proof}
By the definition of the twisted product there is a homomorphism $\pi_B:
{\bf C} \rightarrow {\bf B}$. If ${\bf B}$ is homomorphic to ${\bf T}$ 
then so is ${\bf C}$. In order to prove the converse assume that
there exists a homomorphism $\varphi: {\bf C} \rightarrow {\bf T}$.
Let us define the mapping $\xi: B \rightarrow T$ in the following
way. For an element $b \in B$ let $\xi(b)$ be one of the elements of
$T$ such that $|\pi_B^{-1}(b) \cap \varphi^{-1}(\xi(b))| \geq
\frac{|A|}{|T|}$. We will show that $\xi$ is a homomorphism. Let $R$
be an $r$-ary relational symbol in $\tau$, $b=(b_1,\dots ,b_r) \in
R({\bf B})$. We need to show that $(\xi(b_1), \dots ,\xi(b_r)) \in
R({\bf T})$.

Set $S_i=\varphi^{-1}(\xi(b_i)) \cap \pi_B^{-1}(b_i)$. We succeed to
show that there is a tuple $(c_1, \dots ,c_r) \in R({\bf C})$ with
$c_i \in S_i$: In this case the tuple $(\xi(b_1), \dots
,\xi(b_r)) = (\varphi(c_1), \dots ,\varphi(c_r))$ would be in
$R({\bf T})$, since $\varphi$ is a homomorphism.

Denote the bijections corresponding to $b$ determining the twisted
product ${\bf C}$ by $\alpha_{b,i}: A \rightarrow \pi_B^{-1}(b_i)$.
The tuple $(c_1, \dots ,c_r)$ (where $c_i \in \pi_B^{-1}(b_i)$ for
every $i$) is in $R({\bf C})$ iff $\big( \alpha^{-1}_{b,1}(c_1),
\dots ,\alpha^{-1}_{b,r}(c_r) \big) \in R({\bf A})$.

We use the expander property of ${\bf A}$ for the sets
$\alpha^{-1}_{b,i}(S_i)$ for $1 \leq i \leq r$. Since $R({\bf
C})(S_1, \dots ,S_l)=R({\bf A})(\alpha^{-1}_{b,1}(S_1),\dots
,\alpha^{-1}_{b,r}(S_r))$ we have

$$\big| R({\bf C})(S_1, \dots ,S_l) -
|R({\bf A})| \frac{\prod_{i=1}^l |S_i|}{|A|^l} \big| \leq \varepsilon
|R({\bf A})|.$$

On the other hand $|R({\bf A})| \frac{\prod_{i=1}^l |S_i|}{|A|^l}>
\varepsilon |R({\bf A})|$ by the choice of the sets $S_i$ and
$\varepsilon$. Hence $R({\bf C})(S_1, \dots ,S_r)>0$, there exists
an appropriate tuple $(c_1, \dots ,c_r) \in R({\bf C})$. This
completes the proof of the lemma.
\end{proof}

\begin{proof} (of Theorem~\ref{reduction})
Let us choose $r$ such that every relational symbol in $\tau$ is at
most $r$-ary. Consider a $\frac{1}{t^r+1}$-expander ${\bf A}$ with
girth $>k$ and bounded degree. Hence if $|A|$ is large enough then
$|A|^{\frac{1}{k}} > \Delta({\bf A}) \Delta({\bf S})$ holds. Such an
expander ${\bf A}$ can be constructed in polynomial time (of $|S|$) 
for fixed $t$ and $k$. Now we can use Lemma~\ref{karcsu} for
${\bf A=A}$ and ${\bf B=S}$ to construct a twisted product ${\bf C}$
of girth at least $k$. Set ${\bf S'=C}$. Lemma~\ref{exphely} implies
Theorem~\ref{reduction}.
\end{proof}

\end{document}